\documentclass[11pt]{article}

\usepackage[margin=1in]{geometry}
\usepackage{amsmath,amssymb,amsthm,mathtools}
\usepackage{enumitem}
\usepackage[colorlinks=true,linkcolor=blue,citecolor=blue,urlcolor=blue]{hyperref}
\usepackage{microtype}
\usepackage{natbib}
\newtheorem{theorem}{Theorem}[section]
\newtheorem{proposition}[theorem]{Proposition}
\newtheorem{lemma}[theorem]{Lemma}
\newtheorem{corollary}[theorem]{Corollary}
\newtheorem{remark}[theorem]{Remark}

\DeclareMathOperator{\Var}{Var}
\DeclareMathOperator{\Lip}{Lip}
\DeclareMathOperator{\im}{im}
\DeclareMathOperator{\Id}{Id}

\newcommand{\E}{\mathbb E}

\newcommand{\1}{\mathbf 1}
\newcommand{\norm}[1]{\left\lVert #1\right\rVert}
\newcommand{\ip}[2]{\left\langle #1,#2\right\rangle}

\title{Spectral Gap for the Binary Fixed-Margin Swap Chain}
\author{Weibo Fu \thanks{\href{mailto:wfu@math.princeton.edu}{wfu@math.princeton.edu}}\and Qian Qin \thanks{\href{mailto:qqin@umn.edu}{qqin@umn.edu}} \and Guanyang Wang \thanks{\href{mailto:guanyang.wang@rutgers.edu}{guanyang.wang@rutgers.edu}}}

\begin{document}
	\maketitle
	
	\begin{abstract}
		We prove an explicit spectral-gap lower bound for the lazy swap chain on binary matrices with prescribed row and column sums. This chain is a standard sampler for fixed-margin null models in ecology, statistics, and network analysis. Kannan, Tetali, and Vempala (KTV) conjectured that it mixes rapidly for all feasible margins \citep{kannan1997simple}. We show that for every feasible set of margins on an $m\times n$ binary matrix, the lazy swap chain has spectral gap at least
		$$
		\binom{m}{2}^{-1}\binom{n}{2}^{-1}.
		$$
		The bound is tight in the worst case. Thus, our result proves this KTV conjecture in a stronger quantitative form. The same spectral-gap bound also
		verifies the Mihail--Vazirani conjecture
		\citep{mihail1992expansion,kaibel2004expansion} for fixed-margin
		\(0/1\)-matrix polytopes.
		
		The proof gives a new route to fixed-margin sampling that avoids stability assumptions and canonical-path constructions. We compare the swap chain with a two-row heat-bath chain and use a local-to-global spectral reduction to reduce the analysis from arbitrary $m\times n$ matrices to a three-row problem. The remaining three-row inequality is then proved by separating the scalar column-count sector from the non-scalar Johnson harmonic sectors.

		The proof itself was generated by ChatGPT 5.5 Pro. ChatGPT proposed the whole proof strategy, including the comparison with the two-row heat-bath chain, the reduction to the three-row case, and the decomposition of the three-row function space into the count sector and the Johnson harmonic sectors. It also generated all the technical lemmas and initial proofs. The author's role was to pose the problem, guide the search direction, evaluate the AI-generated arguments, rewrite the proof, and take responsibility for the final form and validity of the result. The full proof of the main theorem has been formalized in Lean, and the
  accompanying formalization is available at the anonymous repository
  \url{https://github.com/guanyangwang/ktv-swap-lean}.		
		
	\end{abstract}
	
	\tableofcontents
	\section{Introduction}
	Binary matrices are a natural way to encode binary relations between two sets of objects. For example, in ecology, one may record which bird species are observed on which islands by an $m\times n$ binary matrix $A$, where $A_{ij}=1$ if species $i$ is present on island $j$, and $A_{ij}=0$ otherwise.
	
	A basic task is to generate uniformly random binary matrices with prescribed row and column sums. This problem arises naturally in ecological null models. Given an observed species--island matrix, ecologists ask whether its pattern reflects real ecological structure, such as competition between species, or can be explained by random chance after controlling for species prevalence and island richness. This leads to the fixed-margin null model, in which one samples uniformly from all binary matrices with the same row and column sums as the observed matrix. The same sampling problem also appears in testing the Rasch model for binary response data, where conditioning on the row and column sums makes all feasible matrices equally likely \citep{rasch1993probabilistic,besag1989generalized}. Related fixed-margin sampling problems arise in network reconstruction, economics, and the social sciences; see \citet{neal2024pattern} for a survey.
	
	The swap algorithm is the most basic Markov chain for this sampling problem. Starting from a binary matrix with the prescribed row and column sums, the algorithm repeatedly chooses two rows and two columns uniformly at random. If the resulting $2\times 2$ submatrix is of the form
	
	\[
	\begin{pmatrix}1&0\\0&1\end{pmatrix}
	\quad\text{or}\quad
	\begin{pmatrix}0&1\\1&0\end{pmatrix},
	\]
	then the algorithm replaces it by the other pattern; otherwise, it leaves the matrix unchanged.
	
The swap algorithm has a long history in statistics and applied sampling. Early work on random $0/1$-matrices with fixed marginals already used swap-based Markov chain Monte Carlo methods for exact conditional inference \citep{rao1996markov}. Despite this long history and wide use, the algorithm has resisted a complete worst-case mixing analysis. A basic open problem has been to determine how many steps are needed for the chain to approach the uniform distribution, or equivalently, to bound its mixing time uniformly over all feasible margins. In 1997, \citet{kannan1997simple} formulated this problem as a general conjecture, predicting that the swap chain mixes in polynomial time for every feasible choice of row and column sums. They also proved a polynomial mixing bound in the regular bipartite case. This conjecture has motivated a long line of work on swap or switch chains for fixed-margin binary matrices and closely related prescribed-margin models.

Progress has taken two somewhat different forms. One line of work proves sharp or near-sharp estimates under strong structural assumptions. In the regular bipartite case, \citet{kannan1997simple} proved a polynomial mixing bound. For sparse regular bipartite margins, \citet{dyer2021sampling} gave an improved mixing bound, and \citet{tikhomirov2023sharp} later proved sharp Poincar'e and log-Sobolev inequalities for the switch chain in a sparse regular regime.  A subsequent work of \citet{tikhomirov2024regularized}
	proved a sharp modified log-Sobolev estimate for fixed degree, giving
	\(O_d(n\log n)\) total-variation mixing.  A second line of work gives rapid-mixing results for broader classes of non-regular margins under suitable structural or stability assumptions. Early progress beyond the regular case includes the work of \citet{miklos2013towards}, who proved rapid mixing for semi-regular bipartite margins. Further results developed stability-based, structural, and comparison methods for larger families of irregular margins and related models, including bounded irregular bipartite and directed degree sequences \citep{erdHos2018efficiently,carstens2018speeding,erdHos2018new}. One major framework along this line is P-stability, introduced by \citet{jerrum1990fast}, which has led to rapid-mixing results for several large classes of margins \citep{amanatidis2019rapid,erdHos2022mixing,gao2021mixing}. These contributions form much of the current theoretical basis for the swap chain, but they still leave open the original conjecture for arbitrary row and column sums.

	The same fixed-margin state spaces have also been studied from counting and algorithmic-sampling viewpoints. On the enumeration side, one asks how many binary matrices have prescribed row and column sums, and studies the typical structure of a uniformly random matrix from this set; sharp asymptotic formulas are known in several sparse, dense, and irregular regimes \citep{greenhill2006asymptotic,canfield2008asymptotic,barvinok2010number}. On the algorithmic side, the general approximate counting and sampling problem is already known to admit polynomial-time randomized algorithms for arbitrary margins: one route goes through reductions to perfect matchings and the permanent, together with the algorithm of \citet{jerrum2004polynomial}, while \citet{bezakova2007sampling} gave a more direct simulated-annealing algorithm for binary contingency tables with prescribed row and column sums. These works show that the fixed-margin space can be sampled efficiently by more global algorithms. They do not, however, settle the KTV conjecture, which asks whether the natural local swap chain itself mixes rapidly for every feasible choice of margins.

	This paper proves a worst-case optimal spectral-gap bound for the swap chain on binary matrices with prescribed margins. For every feasible set of row and column sums on an $m\times n$ binary matrix, the lazy swap chain has spectral gap at least
	\[
	\binom{m}{2}^{-1}\binom{n}{2}^{-1}.
	\]
	This bound is tight in the worst case. In particular, it implies polynomial-time mixing in the matrix dimensions and resolves the binary-matrix case of the conjecture of \citet{kannan1997simple}. 
	In fact, it proves a result stronger than the conjecture itself. 
	
Beyond mixing, our result has a consequence for \(0/1\)-polytopes.  The
Mihail--Vazirani conjecture asserts that the graph of every \(0/1\)-polytope
has edge expansion at least one
\citep{mihail1992expansion,kaibel2004expansion}.  It has been proved for several natural exchange graphs.
Most notably, \citet{anari2024log} proved it for matroid base exchange graphs using
log-concave polynomials and high-dimensional walks. Although the fixed-margin
\(0/1\)-matrix polytopes considered here are not generally matroid base polytopes,
they form another basic exchange family. For fixed margins \(r,c\), let
\[
P_{r,c}
=
\operatorname{conv}\{X\in\{0,1\}^{m\times n}: X\mathbf 1=r,\ X^\top \mathbf 1=c\}
\]
be the associated fixed-margin \(0/1\)-matrix polytope. 
As a corollary of our spectral-gap bound, we prove the Mihail--Vazirani
conjecture for this family of polytopes.  The reason is that the lazy swap
chain assigns the same transition probability to every swap edge, and the
sharp spectral-gap lower bound forces the unweighted swap graph to have edge
expansion at least one.  Since every swap edge is also an edge of
\(P_{r,c}\), the full graph of \(P_{r,c}\) has edge expansion at least one as
well. We note that this conclusion relies on the sharp spectral gap estimate in our theorem, rather than on the polynomial mixing statement predicted by the original KTV conjecture alone.

The proof introduces a new route for analyzing Markov chains of this kind, one that is largely orthogonal to the stability-based and canonical-path methods used in much of the earlier work on swap chains. Rather than constructing multicommodity flows on the full state space, or proving rapid mixing under regularity assumptions on the margins, we compare the swap chain with a two-row heat-bath chain, equivalently the row-pair Curveball algorithm introduced in the ecology literature by \citet{strona2014fast}. We then apply a local-to-global spectral reduction \citep{caputo2008spectral}, which reduces the general $m\times n$ problem to the corresponding problem on three rows. In this way, the analysis of an arbitrary number of rows is reduced to a fixed three-row inequality, although this reduction requires a much sharper spectral-gap estimate in the reduced problem.	The remaining three-row problem is the only genuinely hard case. We solve it by decomposing the function space into a scalar column-count sector and non-scalar Johnson harmonic sectors. The count sector is controlled by a sandwitch path-coupling argument, while the harmonic sectors are controlled through representation-theoretic symmetry and exact Johnson-scheme identities. Thus, the proof replaces the global combinatorial geometry of the full swap graph by a small, highly structured spectral problem.
	
This approach is new in two ways. First, it shows that the worst-case behavior of the fixed-margin swap chain with an arbitrary number of rows is already captured by a three-row inequality. Second, it combines heat-bath comparison, a Caputo-type local-to-global reduction, and Johnson harmonic analysis in a form that does not require stability, density, regularity, or degree-sequence assumptions. The reduction is  close in spirit to the spectral-independence method: in both cases, one proves a global mixing or spectral-gap bound by controlling a family of lower-dimensional conditional problems and then propagating these local estimates back to the full chain \citep{anari2021spectral,blanca2022mixing,feng2022rapid,liu2023spectral}. Here, however, the local object is not an influence matrix, but a concrete three-row Markov chain whose full spectral structure can be analyzed. The method gives more than the original KTV conjecture asks for. It yields the sharp worst-case spectral-gap estimate for the swap chain and, at the same time, a spectral-gap bound for the row-pair Curveball chain. We expect the proof strategy to be useful for other constrained sampling
problems with exchange moves.  A natural next target is the swap chain for
contingency tables, where the states are nonnegative integer matrices with
fixed row and column sums; see, for example, the Markov-basis framework of
\citet{diaconis1998algebraic}.

Finally, we note that binary matrices with fixed row and column sums are equivalently bipartite graphs with fixed degrees. The conjecture of \citet{kannan1997simple} is part of a broader set of questions on switch chains for graph realizations with prescribed degrees. We refer to \citet{cooper2007sampling,greenhill2011polynomial,miklos2013towards,greenhill2014switch,greenhill2018switch,erdHos2018efficiently,amanatidis2019rapid,erdHos2022mixing,gao2021mixing} for important progress on these questions.

	\paragraph{Statement of AI use:} The proof ideas and the initial proofs were generated solely by GPT-5.5 Pro. Prompt history: \url{https://chatgpt.com/share/6a37fe0d-f98c-83ea-a481-f16f68e47ddb}. One author prompted GPT-5.5 Pro with the problem.
	The model first reviewed the known literature and focused on P-stable degree sequences, but this direction got stuck.
	The key turn came when author redirected the search with the prompt: \texttt{Maybe you need to think out of the box. You should consider other related areas instead of just focusing on combinatorics\\ and theoretical computer science literature, as the concept itself has also a clear algebraic structure.} After this redirection, GPT-5.5 Pro proposed the entire proof over a few rounds.

	\section{Main Result}
	
	Let
	\[
	\Omega(r,c)=
	\left\{X\in\{0,1\}^{m\times n}:\
	\sum_{j=1}^n X_{ij}=r_i\ \forall i,
	\quad
	\sum_{i=1}^m X_{ij}=c_j\ \forall j
	\right\}.
	\]
	Assume throughout that \(\Omega(r,c)\neq\varnothing\), a feasibility condition characterized by the classical Gale--Ryser theorem \citep{gale1957theorem,ryser1957combinatorial}.
	
	The lazy swap chain is the following reversible Markov chain on \(\Omega(r,c)\).  With probability \(1/2\) it stays put.  Otherwise it chooses an unordered pair of distinct rows \(\{a,b\}\) and an unordered pair of distinct columns \(\{u,v\}\), both uniformly.  If the resulting \(2\times2\) submatrix is
	\[
	\begin{pmatrix}1&0\\0&1\end{pmatrix}
	\quad\text{or}\quad
	\begin{pmatrix}0&1\\1&0\end{pmatrix},
	\]
	it replaces it by the other one.  If the chosen submatrix is not switchable, the chain stays put.  Let \(P_{\rm sw}\) denote its transition matrix. We use the following convention for the spectral gap. If
	\(|\Omega(r,c)|\ge 2\), then
	\[
	\gamma_{\rm sw}=1-\lambda_2(P_{\rm sw}),
	\]
	where \(\lambda_2(P_{\rm sw})\) is the second largest eigenvalue of \(P_{\rm sw}\).
	If \(|\Omega(r,c)|=1\), we set
	\[
	\gamma_{\rm sw}=1.
	\]

	The following result shows that the lazy chain is irreducible and aperiodic. Irreducibility is equivalent to the statement that any two binary matrices with the same row and column sums can be connected by a sequence of swap moves, which follows from \citet{ryser1957combinatorial}. Aperiodicity follows from the laziness of the chain.

	\begin{proposition}[Connectivity of the swap chain]\label{prop:connectivity}
		For every feasible pair of margins \((r,c)\), the non-lazy swaps connect \(\Omega(r,c)\).  Consequently the lazy swap chain is irreducible and aperiodic on \(\Omega(r,c)\).
	\end{proposition}

	Our main result is the following spectral-gap bound for the lazy swap chain.
	
	\begin{theorem}[Main theorem]\label{thm:main}
		For every feasible pair of margins \((r,c)\), the lazy swap chain satisfies
		\[
		\gamma_{\rm sw}\ge \binom{m}{2}^{-1}\binom{n}{2}^{-1}
		\]
		whenever \(m,n\ge2\).  Consequently,
		\[
		t_{\rm mix}(\varepsilon)
		\le \binom{m}{2}\binom{n}{2}
		\left(\log |\Omega(r,c)|+\log \varepsilon^{-1}\right)
		\le \binom{m}{2}\binom{n}{2}
		\left(mn\log2+
		\log \varepsilon^{-1}\right).
		\]
	\end{theorem}
	
	\begin{remark}[Sharpness of the constant]
		The bound in Theorem~\ref{thm:main} is sharp in the worst case. Indeed, fix two rows and two columns, say
		rows \(1,2\) and columns \(1,2\), and take margins
		\[
		r_1=r_2=1,\qquad r_3=\cdots=r_m=0,
		\]
		and
		\[
		c_1=c_2=1,\qquad c_3=\cdots=c_n=0.
		\]
		Then the state space consists of exactly the two matrices whose only
		nonzero entries form one of the two \(2\times 2\) permutation matrices
		inside rows \(1,2\) and columns \(1,2\). Thus the unique nontrivial move is
		the swap on this single checkerboard block.
		
		For the lazy swap chain, from either state the probability of moving to
		the other state is
		\[
		p=\frac12\binom m2^{-1}\binom n2^{-1},
		\]
		because the chain must first avoid the lazy step and then choose the
		unique active row pair and column pair. Hence the transition matrix on this
		two-point space is
		\[
		\begin{pmatrix}
			1-p & p\\
			p & 1-p
		\end{pmatrix},
		\]
		whose spectral gap is \(2p\). Therefore
		\[
		\gamma_{\mathrm{sw}}
		=
		\binom m2^{-1}\binom n2^{-1}.
		\]
		Thus no larger uniform lower bound, even by a constant factor in this
		normalization, can hold for all feasible margins.
	\end{remark}
	
	For a finite graph \(G=(V,E)\) and a set \(U\subseteq V\), write
	\[
	\partial_G U
	=
	\{\{x,y\}\in E:x\in U,\ y\notin U\}.
	\]
	We say that \(G\) has edge expansion at least one if
	\[
	|\partial_G U|\ge |U|
	\]
	for every \(U\subseteq V\) with \(0<|U|\le |V|/2\).
	
	\begin{corollary}[Mihail--Vazirani for fixed-margin \(0/1\)-matrix polytopes]
		\label{cor:MV-fixed-margin}
		For every feasible pair of margins \((r,c)\), let
		\[
		P_{r,c}
		=
		\operatorname{conv}\{X\in\{0,1\}^{m\times n}: X\mathbf 1=r,\ X^\top \mathbf 1=c\}.
		\]
		Then the graph of \(P_{r,c}\) has edge expansion at least one.  Equivalently,
		for every set \(U\) of vertices of \(P_{r,c}\) with
		\(0<|U|\le |V(P_{r,c})|/2\), at least \(|U|\) edges of the graph of
		\(P_{r,c}\) leave \(U\).  Thus the Mihail--Vazirani conjecture holds for
		fixed-margin \(0/1\)-matrix polytopes.
	\end{corollary}
	
	\section{Comparison to a Two-row Heat-bath Chain}
	
	\subsection{The two-row heat-bath chain}

	We compare $\gamma_{\rm sw}$ to the spectral gap of the following two-row heat-bath chain targeting the uniform distribution on $\Omega(r,c)$.
	Given the current binary matrix $X \in \Omega(r,c)$, choose an unordered pair \(\{a,b\} \subset [m]\) uniformly, where $a \neq b$; freeze all rows except rows \(a,b\); then resample rows \(a,b\) uniformly among all completions satisfying $\sum_{j=1}^n X_{aj} = r_a$, $\sum_{j=1}^n X_{bj} = r_b$, and $X_{aj} + X_{bj} = c_j - \sum_{i \not\in\{a,b\}} X_{ij}$ for $j \in \{1,\dots,n\}$. 
	
	Let \(P_H\) be its transition matrix and let
	\[
	\gamma_H=1-\lambda_2(P_H).
	\]
	Then $P_H$ corresponds to a Gibbs algorithm that updates two coordinates (rows) at a time, while $P_{\rm sw}$ corresponds to a hybrid version of this Gibbs algorithm, with the resampling of the two selected rows replaced by a possible column swap.
	We use a decomposition technique \citep{qin2024spectral} to compare their spectral gaps.
	For similar techniques, see, e.g., \cite{caracciolo1992two,madras2002markov,ge2018simulated,andrieu2018uniform}, and the literature on spectral independence \citep{anari2021spectral,blanca2022mixing,feng2022rapid,liu2023spectral}.
	
	\subsection{Spectral gap comparison}
	
	For a matrix $X \in \Omega(r,c)$ and $\{a,b\} \subset [m]$, let $X_{-\{a,b\}}$ be~$X$ excluding rows $a,b$.
	The transition matrix $P_H$ may be written as
	\[
	P_H(X, X') = {m \choose 2}^{-1} \sum_{\substack{\{a,b\} \subset [m] \\ a \neq b}} \pi^{a,b}(X'
	\mid X_{-\{a,b\}} ),
	\]
	where $\pi^{a,b}(\cdot \mid z)$ is the uniform distribution on $\{X \in \Omega(r,c): \, X_{-\{a,b\}} = z\}$.
	
	The transition matrix $P_{\rm sw}$ can be written into a similar form.
	For $\{a,b\} \subset [m]$ and a value of $X_{-\{a,b\}}$ denoted by~$z$, let $P^{a,b}_{\rm sw}(\cdot, \cdot \mid z)$ be the transition matrix of the following Markov chain on the support of $\pi^{a,b}(\cdot \mid z)$:
	Given the current state $X \in \Omega(r,c)$ satisfying $X_{-\{a,b\}} = z$, with probability $1/2$, stay put.
	Otherwise, uniformly select a pair of distinct indices $(u,v)$ from $[n]$; if 
	\[
	\begin{pmatrix} X_{au} & X_{av} \\X_{bu}& X_{bv}\end{pmatrix} = \begin{pmatrix}1&0\\0&1\end{pmatrix}
	\; \text{or} \;
	\begin{pmatrix}0&1\\1&0\end{pmatrix},
	\]
	replace the submatrix with the other one. 
	Then
	\[
	P_{\rm sw}(X, X') = {m \choose 2}^{-1} \sum_{\substack{\{a,b\} \subset [m] \\ a \neq b}} P^{a,b}_{\rm sw}(X, X' \mid X_{-\{a,b\}} ).
	\]
	
	We compare the two Markov chains in terms of Dirichlet forms.
	For a reversible Markov chain with transition matrix \(M\) and stationary
	measure \(\varpi\), and a function $f$, define the Dirichlet form
	\[
	\mathcal E_M(f,f)
	=
	\frac12\sum_{x,y}\varpi(x) \, M(x,y) \, [f(x)-f(y)]^2 .
	\]
	
	\begin{lemma}[Comparison with heat bath]\label{lem:two-column}
		We have
		\[
		\gamma_{\rm sw} \ge \binom n2^{-1}\gamma_H .
		\]
	\end{lemma}
	
	\begin{proof}

		Fix a pair of indices $\{a,b\} \subset [m]$ where $a \neq b$.
		Consider a matrix~$X$ in $\Omega(r,c)$ with $X_{-\{a,b\}}$ frozen at some value $z$.
		For each
		column \(j\), the value \(X_{aj}+X_{bj}\) is then fixed. Columns with value
		\(0\) or \(2\) are forced, while columns with value \(1\) each have two possible configurations.
		Let
		\[
		T=\{j:X_{aj}+X_{bj}=1\}.
		\]
		The number of $j$'s in \(T\) such that $X_{aj} = 1$ is fixed due to the restriction $\sum_{j=1}^n X_{aj} = r_a$; call it \(k\).
		Thus we may represent rows $a,b$ as a set
		\[
		S\subseteq T,\qquad |S|=k,
		\]
		where \(j\in S\) means \(X_{aj}=1\) and \(X_{bj}=0\). Let \(\nu\) be the
		uniform measure on the collection of all such \(S\), and write \(g(S)\) for the corresponding
		value of \(f\).

		Having frozen all other rows, the two-row heat-bath $P_H$ update replaces the current value of \(S\) by an independent uniform element
		of \(\{S' \subseteq T:|S'|=k\}\). Hence its conditional Dirichlet form is
		\[
		\begin{aligned}
			\mathcal E_{H}^{a,b}(f,f \mid z) :=& \frac{1}{2} \sum_{X,X'} \pi^{a,b}(X \mid z) \, \pi^{a,b}(X' \mid z) \, [f(X) - f(X')]^2 \\
			=& \frac{1}{2} \sum_{S,S'} \nu(\{S\}) \, \nu(\{S'\}) \, [g(S) - g(S')]^2 = \operatorname{Var}_{\nu}(g),
		\end{aligned}
		\]
		We write $\nu(\{S\})$ instead of $\nu(S)$ to emphasize $\nu$ is a distribution on the $k$-subsets of~$T$ as opposed to~$T$ itself.
		
		On the other hand, $P_{\rm sw}^{a,b}(\cdot, \cdot \mid z)$ can be reformulated as a conditional swap chain that is reversible with respect to~$\nu$.
		We will denote by $P_{\rm csw}$ the transition matrix of the conditional swap chain described below:
		With probability $1/2$, this local chain stays put, and otherwise it updates the current state~$S$ as follows: two distinct elements $u,v \in [n]$ are chosen uniformly at random, and if one belongs to~$S$ and the other belongs to $T \setminus S$, the former is swapped out of~$S$ for the latter.
		Thus, a particular transition
		\[
		S\longmapsto (S\setminus\{u\}) \cup\{v\},
		\qquad u\in S,\quad v \in T \setminus S,
		\]
		occurs with probability
		\[
		p=\frac{1}{2\binom n2},
		\]
		where the factor \(1/2\) is the laziness.

		Consider two adjacent local states
		\[
		S=C\cup\{u\},\qquad S'=C\cup\{u'\},
		\]
		where $u \neq u'$ and $u,u' \not\in C$.
		We couple one step of the two chains as follows. With probability \(p\), let
		the chain from \(S\) jump to \(S'\) and let the chain from \(S'\) stay put.
		With probability \(p\), let the chain from \(S'\) jump to \(S\) and let the
		chain from \(S\) stay put. In either case the two chains coalesce. All
		remaining transition mass can be coupled so that the two resulting states remain adjacent or become identical:
		when the chain from $S$ jumps to $C \cup \{v\}$ where $v \not\in C \cup \{u,u'\}$, so does the chain from $S'$;
		when the chain from $S$ jumps to $(C \setminus \{v\}) \cup \{u,u'\}$ where $v \in C$, so does the chain from $S'$;
		when the chain from $S$ jumps to $(C \setminus \{v\}) \cup \{u,w\}$ where $v \in C$ and $w \not\in C \cup \{u,u'\}$, the chain from $S'$ jumps to $(C \setminus \{v\}) \cup \{u',w\}$.
		
		Define the graphical distance
		\[
		d(S,S')=|S\setminus S'|,
		\]
		Then, whenever \(S,S'\) are adjacent, if $(S_1,S_1')$ is generated by one step of the above coupled chain,
		\[
		\mathbb E[d(S_1,S'_1)]
		\le
		1-2p
		=
		1-\binom n2^{-1}.
		\]
		By path coupling \citep{bubley1997path}, the conditional lazy swap chain has spectral gap at least
		\(\binom n2^{-1}\). Therefore,
		\[
		\begin{aligned}
			\mathcal E_{\rm sw}^{a,b}(f,f \mid z) := & \frac{1}{2} \sum_{X, X'} \pi^{a,b}(X \mid z) \, P^{a,b}_{\rm sw}(X, X' \mid z) \, [f(X) - f(X')]^2 \\
			= & \frac{1}{2} \sum_{S,S'} \nu(\{S\}) \, P_{\rm csw}(S, S') \, [g(S) - g(S')]^2  \\
			\ge&
			\binom n2^{-1}\operatorname{Var}_{\nu}(g).
		\end{aligned}
		\]
		Since \(\operatorname{Var}_{\nu}(g)\) is exactly the conditional Dirichlet
		form of the two-row heat-bath update, we get
		\[
		\mathcal E_{\rm sw}^{a,b}(f,f \mid z)
		\ge
		\binom n2^{-1}
		\mathcal E_{H}^{a,b}(f,f \mid z) .
		\]
		Averaging over the conditioned $z$ and over the uniformly chosen row
		pair \(\{a,b\}\), we obtain
		\[
		\mathcal E_{\rm sw}(f,f)
		\ge
		\binom n2^{-1}\mathcal E_H(f,f).
		\]
		Dividing by \(\operatorname{Var}_{\pi}(f)\), where \(\pi\) is the uniform
		measure on matrices with the prescribed row and column sums, and taking the
		infimum over all nonconstant \(f\), gives
		\[
		\gamma_{\rm sw}\ge \binom n2^{-1}\gamma_H .
		\]

	\end{proof}
	
	\section{Reduction to Three Rows}
	
	The following reduction is a special case of the projection argument of
	\citet{caputo2008spectral}. We give a self-contained proof for the row-pair
	heat-bath chain on binary matrices with fixed row and column sums.
	
	Recall \(\Omega=\Omega(r,c)\) is the set of \(m\times n\) binary matrices with
	prescribed row sums \(r\) and column sums \(c\), and let \(\pi\) be the uniform
	measure on \(\Omega\). For a pair of row indices \(\theta=\{a,b\} \subset [m]\) where $a \neq b$, let \(E_{\theta}\) denote
	conditional expectation given all rows outside \(\theta\). 
	In other words, 
	\[
	E_{\theta}f(X) = \sum_{X'} f(X') \, \pi^{a,b}(X' \mid X_{-\{a,b\}}).
	\]
	Put
	\[
	R_{\theta}=\Id-E_{\theta} .
	\]
	Since \(E_{\theta}\) is a conditional expectation, it is an orthogonal projection on $L^2(\pi)$, and
	so \(R_{\theta}\) is also an orthogonal projection.
	
	The row-pair heat-bath transition matrix is
	\[
	P_H
	=
	\binom m2^{-1}
	\sum_{\theta\in\binom{[m]}2}E_{\theta} .
	\]
	Equivalently,
	\[
	P_H
	=
	\Id-\binom m2^{-1}Q_H,
	\qquad
	Q_H=\sum_{\theta\in\binom{[m]}2}R_{\theta} .
	\]
	Moreover, the row-pair heat-bath chain is irreducible because its transition contains the usual \(2\times2\) swap chain, which is connected on \(\Omega(r,c)\) by Proposition \ref{prop:connectivity}; it is aperiodic because each heat-bath update has positive probability of returning the current state. Therefore, to prove
	\[
	\gamma_H\ge \binom m2^{-1},
	\]
	it is enough to show that every positive eigenvalue of \(Q_H\) is at least
	\(1\).
	
	\begin{lemma}[Three-row reduction]
		\label{lem:three-row-reduction}
		Assume that for every triple of rows \(T\subseteq[m]\), and every
		positive-probability fixing of the rows outside \(T\), the row-pair heat-bath
		chain on the remaining three rows has spectral gap at least \(1/3\). Then
		\[
		\gamma_H\ge \binom m2^{-1}.
		\]
	\end{lemma}
	
	\begin{proof}
		For a test function \(f\), define
		\[
		a_{\theta,\theta'}=\ip{R_\theta f}{R_{\theta'}f},
		\qquad
		\mathcal A=\sum_\theta a_{\theta,\theta}
		=
		\sum_\theta \norm{R_\theta f}_2^2 .
		\]
		where $\langle f, g \rangle = \sum_X \pi(X) \, f(X) \, g(X)$, and $\|f\|_2 = \sqrt{\langle f, f \rangle}$.
		Since \(R_\theta\) is an orthogonal projection,
		\[
		\ip{f}{R_{\theta}f}
		=
		\ip{R_{\theta}f}{R_{\theta}f}
		=
		\norm{R_{\theta}f}_2^2 .
		\]
		Therefore
		\[
		\ip{f}{Q_H f}=\mathcal A,
		\qquad
		\norm{Q_H f}_2^2=\sum_{\theta,\theta'}a_{\theta,\theta'} .
		\]

		We first  claim the following commutation fact. If
		\(\theta\cap \theta'=\varnothing\), then
		\[
		E_{\theta}E_{\theta'}=E_{\theta'}E_{\theta} .
		\]
		Indeed, condition on all rows outside \(\theta\cup \theta'\). During the \(\theta\)-update,
		the two rows in \(\theta'\) are held fixed, so the column-wise sum of the two rows
		in \(\theta\) is fixed. Similarly, during the \(\theta'\)-update, the column-wise sum of
		the two rows in \(\theta'\) is fixed. Once these two column-wise pair sums are fixed,
		the choices for the two row pairs are independent uniform choices, and averaging
		over one pair then the other is independent of the order. Hence the two
		conditional expectations commute.
		
		Thus \(R_{\theta}=\Id-E_{\theta}\) and \(R_{\theta'}=\Id-E_{\theta'}\) are commuting orthogonal
		projections. Hence \(R_\theta R_{\theta'}\) is positive semidefinite, and
		\[
		a_{\theta,\theta'}
		=
		\ip{R_{\theta}f}{R_{\theta'}f}
		=
		\ip{f}{R_{\theta}R_{\theta'}f}
		\ge 0
		\]
		whenever \(\theta\cap \theta'=\varnothing\).
		
		Now fix a triple of row indices \(T\subseteq[m]\), and define
		\[
		Q_H^{(T)}=\sum_{\substack{\theta\subseteq T\\ |\theta|=2}}R_{\theta} .
		\]
		After conditioning on $X_{-T}$, the rows outside \(T\), the three-row heat-bath transition
		matrix is
		\[
		P_T
		=
		\frac13\sum_{\substack{\theta \subseteq T\\ |\theta|=2}}E_{\theta}
		=
		\Id-\frac13 Q_H^{(T)} .
		\]
		By assumption, this conditional chain has spectral gap at least \(1/3\). Hence
		every positive eigenvalue of \(Q_H^{(T)}\) is at least \(1\). 
		Then, if $X$ is distributed as~$\pi$,
		\[
		\mathbb{E}\{[Q_H^{(T)} f(X)]^2 \mid X_{-T}\} \geq \mathbb{E}[f(X) \, Q_H^{(T)} f(X) \mid X_{-T}].
		\]
		Taking the average yields
		\[
		\sum_{\substack{\theta,\theta'\subseteq T\\ |\theta|=|\theta'|=2}}a_{\theta,\theta'}
		\ge
		\sum_{\substack{\theta\subseteq T\\ |\theta|=2}}a_{\theta,\theta}.
		\]
		
		Sum this inequality over all triples \(T\subseteq[m]\). Each diagonal term
		\(a_{\theta,\theta}\) is counted in exactly \(m-2\) triples. Each ordered off-diagonal
		term \(a_{\theta,\theta'}\) with \(\theta\cap \theta'\ne\varnothing\) is counted exactly once,
		namely in the triple \(\theta\cup \theta'\). Disjoint pair-pairs are not counted. Thus
		\[
		(m-2) \mathcal A
		+
		\sum_{\substack{\theta\ne \theta'\\ \theta\cap \theta'\ne\varnothing}}a_{\theta,\theta'}
		\ge
		(m-2) \mathcal A .
		\]
		Therefore
		\[
		\sum_{\substack{\theta\ne \theta'\\ \theta\cap \theta'\ne\varnothing}}a_{\theta,\theta'}\ge0 .
		\]
		Together with the nonnegativity of the disjoint-pair terms, this gives
		\[
		\sum_{\theta,\theta'}a_{\theta,\theta'}
		\ge
		\sum_{\theta} a_{\theta,\theta}
		=
		\mathcal A.
		\]
		Equivalently,
		\[
		\norm{Q_H f}_2^2\ge \ip{f}{Q_H f}.
		\]
		
		Since \(Q_H \) is self-adjoint and nonnegative, diagonalizing \(Q_H \) shows that the
		last inequality for all \(f\) is equivalent to saying that every positive
		eigenvalue of \(Q_H \) is at least \(1\). Finally,
		\[
		P_H=\Id-\binom m2^{-1}Q_H ,
		\]
		so the row-pair heat-bath chain has spectral gap at least
		\[
		\binom m2^{-1}.
		\]
	\end{proof}
	
	Thus, to prove the row-pair heat-bath bound for all \(m\ge3\), it remains to
	show that every row-pair heat-bath chain on \(3\times n\) binary matrices with
	arbitrary feasible margins has spectral gap at least \(1/3\).

	\section{Analysis of the Three-row Model}
	
	\subsection{Overview}
	
	In this section, we study the three-row heat-bath chain.
	Since the analysis is quite involved, it helps to introduce a new set of notations.

	Consider a conditional three-row state space obtained after fixing all rows
	outside a chosen triple. Columns whose residual column sum is \(0\) or \(3\)
	are forced and may be deleted. Let \(P\) be the set of remaining columns with
	residual column sum \(1\), and let \(Q\) be the set of remaining columns with
	residual column sum \(2\).  Put \(p=|P|\) and \(q=|Q|\).
	
	A state is encoded by two ordered partitions
	\[
	A_1\sqcup A_2\sqcup A_3=P,
	\qquad
	B_1\sqcup B_2\sqcup B_3=Q,
	\]
	where \(A_i\) is the set of sum-one columns occupied by row \(i\), and \(B_i\) is the set of sum-two columns missed by row \(i\).  The row-sum constraints are equivalent, after absorbing forced columns into constants, to
	\[
	|A_i|-|B_i|=d_i,
	\qquad i=1,2,3,
	\]
	where \(d_1+d_2+d_3=p-q\).  The stationary measure of the three-row heat-bath is uniform over all labeled states satisfying these constraints.
	
	In view of the stationary measure, let
	\[
	\mathcal F_i=\sigma(A_i,B_i),
	\qquad
	P_i=\E[\,\cdot\mid\mathcal F_i],
	\qquad
	K_3=\frac13(P_1+P_2+P_3).
	\]
	Updating the pair of rows different from \(i\) is the same as conditioning on \(\mathcal F_i\), so the pair heat-bath chain on three rows is exactly \(K_3\).
	
	\begin{theorem}[Three-row projection inequality]\label{thm:three-row}
		For every feasible conditional three-row state space,
		\[
		\operatorname{gap}(K_3)\ge\frac13.
		\]
		Equivalently, for every mean-zero function \(f\),
		\[
		\sum_{i=1}^3 \Var\left(\E[f\mid\mathcal F_i]\right)
		\le 2\Var(f).
		\]
	\end{theorem}
	
	The next subsections prove Theorem~\ref{thm:three-row}.
	
	\subsection{The scalar count sector}
	
	First consider test functions depending only on the row-count vector.  We use the vector
	\[
	(|B_1|,|B_2|,|B_3|),
	\]
	which is equivalent to using the corresponding \(A\)-counts.
	
	\begin{lemma}[One-dimensional sandwich coupling]\label{lem:sandwich}
		Let \(u,v\) be positive log-concave weights on \(\mathbb Z_{\ge0}\), and define
		\[
		\mu_s(t)=\frac{u(t)v(s-t)}{\sum_a u(a)v(s-a)}
		\]
		on its natural support.  If \(T_s\sim\mu_s\) and \(T_{s+1}\sim\mu_{s+1}\), then there is a coupling such that
		\[
		T_{s+1}\in\{T_s,T_s+1\}
		\]
		almost surely.
	\end{lemma}
	
	\begin{proof}
		The likelihood ratio
		\[
		\frac{\mu_{s+1}(t)}{\mu_s(t)}\propto
		\frac{v(s+1-t)}{v(s-t)}
		\]
		is increasing in \(t\), because \(v(k+1)/v(k)\) is decreasing in \(k\).  Hence \(\mu_s\preceq_{\rm st}\mu_{s+1}\).  Applying the same argument after swapping the two coordinates shows that \(s+1-T_{s+1}\) stochastically dominates \(s-T_s\), equivalently
		\[
		T_{s+1}\preceq_{\rm st} T_s+1.
		\]
		Thus
		\[
		T_s\preceq_{\rm st} T_{s+1}\preceq_{\rm st}T_s+1.
		\]
		For integer laws, the quantile coupling realizes both inequalities at the same time: if \(U\) is uniform on \([0,1]\) and \(F_\nu^{-1}(U)=\min\{t:F_\nu(t)\ge U\}\), set
		\[
		T_s=F_{\mu_s}^{-1}(U),
		\qquad
		T_{s+1}=F_{\mu_{s+1}}^{-1}(U).
		\]
		The two stochastic inequalities imply \(T_s\le T_{s+1}\le T_s+1\) for every \(U\).  Since the variables are integer-valued, this is the desired coupling.
	\end{proof}
	
	\begin{lemma}[Scalar Bessel inequality]\label{lem:bessel}
		Let
		\[
		\Lambda_N=\{x=(x_1,x_2,x_3)\in\mathbb Z_{\ge0}^3:x_1+x_2+x_3=N\},
		\]
		and let
		\[
		\pi(x)=\frac1Z\prod_{i=1}^3\frac1{x_i!(x_i+\delta_i)!},
		\qquad \delta_i\in\mathbb Z_{\ge0}.
		\]
		Let \(Q_if=\E_\pi[f\mid X_i]\), and let \(K=(Q_1+Q_2+Q_3)/3\).  Then
		\[
		\operatorname{gap}(K)\ge\frac13.
		\]
	\end{lemma}
	
	\begin{proof}
		Put
		\[
		w_i(k)=\frac1{k!(k+\delta_i)!}.
		\]
		The ratios
		\[
		\frac{w_i(k+1)}{w_i(k)}=\frac1{(k+1)(k+\delta_i+1)}
		\]
		are decreasing in \(k\), so each \(w_i\) is log-concave.
		
		Put on \(\Lambda_N\) the graph metric
		\[
		d(x,y)=\frac12\sum_{i=1}^3 |x_i-y_i|.
		\]
		Adjacent states have the form \(y=x-e_a+e_b\).  Couple one step of \(K\) from \(x\) and from \(y\) using the same frozen coordinate.
		
		If coordinate \(c\), the third coordinate different from \(a,b\), is frozen, then \(x_c=y_c\), so the conditional laws of the other two coordinates are identical; couple them identically, and the chains coalesce.  If coordinate \(a\) is frozen, then the two remaining totals differ by one.  Lemma~\ref{lem:sandwich} couples the two conditional resamplings so that the full updated configurations remain adjacent.  The same argument applies when coordinate \(b\) is frozen.  Therefore, given adjacent \(x,y\), if $(X',Y')$ are generated via one step of the coupled chain,
		\[
		\E[d(X',Y')]
		\le \frac13\cdot0+\frac13\cdot1+\frac13\cdot1
		=\frac23.
		\]
		By path coupling \citep{bubley1997path},
		\[
		W_1(K(x,\cdot),K(y,\cdot))\le\frac23 d(x,y)
		\]
		for all \(x,y\).  Kantorovich duality then gives contraction of the Lipschitz seminorm:
		\[
		\Lip(Kf)\le\frac23\Lip(f).
		\]
		If \(Kf=\lambda f\) and \(f\) is nonconstant, then \(\Lip(f)>0\), so
		\[
		|\lambda|\Lip(f)=\Lip(Kf)\le\frac23\Lip(f).
		\]
		Thus every nonconstant eigenvalue has absolute value at most \(2/3\), and in particular \(\operatorname{gap}(K)\ge1/3\).
	\end{proof}
	
	\begin{lemma}[Identification of the count chain]\label{lem:count-id}
		The restriction of \(K_3\) to count-dependent functions is of the form in Lemma~\ref{lem:bessel}.
	\end{lemma}
	
	\begin{proof}
		Let
		\[
		b_i=|B_i|,
		\qquad
		a_i=|A_i|=b_i+d_i.
		\]
		The number of labeled states with these counts is
		\[
		\frac{p!}{a_1!a_2!a_3!}\frac{q!}{b_1!b_2!b_3!},
		\]
		which is proportional, as a function of the counts, to
		\[
		\prod_{i=1}^3\frac1{a_i!b_i!}.
		\]
		If \(d_i\ge0\), set \(x_i=b_i\).  If \(d_i<0\), set \(x_i=a_i=b_i+d_i\).  Then
		\[
		a_i!b_i!=x_i!(x_i+|d_i|)!,
		\]
		and \(x_1+x_2+x_3\) is fixed.  Hence the count law is precisely the Bessel law of Lemma~\ref{lem:bessel}, with \(\delta_i=|d_i|\).  The apparent upper bounds \(a_i\le p\) and \(b_i\le q\) are automatic from nonnegativity and the fixed sums \(\sum_i a_i=p\), \(\sum_i b_i=q\).
		
		For a count-dependent function, conditioning on \(\mathcal F_i=\sigma(A_i,B_i)\) is the same as conditioning on \(|B_i|\), since labels are exchangeable within fixed counts.  Therefore \(K_3\) on count-dependent functions is exactly the chain in Lemma~\ref{lem:bessel}.
	\end{proof}
	
	\subsection{Johnson harmonics and non-scalar sectors}
	
	\subsubsection{Johnson harmonic preliminaries}
	
	Let \(E\) be a finite set with \(|E|=N\).  For \(\ell\ge0\), let \(H_E^\ell\) be the space of functions \(h\) on \(\ell\)-subsets of \(E\) whose lower shadows vanish:
	\[
	\sum_{L\supset T,\ |L|=\ell} h(L)=0
	\quad\text{for every }T\subset E,\ |T|<\ell.
	\]
	For \(a\ge\ell\), define the lift
	\[
	h^{(a)}(A)=\sum_{L\subset A,\ |L|=\ell}h(L),
	\qquad |A|=a,
	\]
	and set \(h^{(a)}=0\) if \(a<\ell\).  For \(\ell=0\), \(H_E^0\) is the space of constants and \(h^{(a)}\) is constant.
	
	This decomposition can be viewed either as the eigenspace decomposition of the
	Johnson association scheme or as the discrete-harmonic decomposition on a
	Boolean slice.  The discrete-harmonic viewpoint goes back at least to
	\citet{delsarte1978hahn}; see also \citet{filmus2016orthogonal} for the
	harmonic multilinear-polynomial representation of functions on a slice.  From the representation-theoretic viewpoint, it is the two-row case of Young's rule for the Young permutation module $M^{(N-a,a)}$; see Theorem~2.11.2 of \citet{sagan2001symmetric}.  We include a self-contained finite-dimensional proof.
	
	\begin{lemma}[Johnson harmonic decomposition]\label{lem:johnson-decomp}
		For each \(a\), the functions on \(a\)-subsets of \(E\) decompose orthogonally as
		\[
		L^2\!\left(\binom{E}{a}\right)
		=\bigoplus_{0\le\ell\le\min(a,N-a)}
		\{h^{(a)}:h\in H_E^\ell\}.
		\]
	\end{lemma}
	
	\begin{proof}
		Write
		\[
		V_r=L^2\!\left(\binom{E}{r}\right).
		\]
		Define the up and down operators
		\[
		U_r:V_r\to V_{r+1},
		\qquad
		(U_rf)(A)=\sum_{\substack{B\subset A\\ |B|=r}}f(B),
		\]
		and
		\[
		D_r:V_r\to V_{r-1},
		\qquad
		(D_rf)(B)=\sum_{\substack{A\supset B\\ |A|=r}}f(A).
		\]
		Then
		\[
		D_{r+1}=U_r^*.
		\]
		The lower-shadow condition is exactly
		\[
		H_E^\ell=\ker D_\ell
		\]
		for \(\ell\ge 1\), while \(H_E^0=V_0\).
		
		We use the elementary identity
		\begin{align}\label{eqn:difference-identity}
			D_{r+1}U_r-U_{r-1}D_r=(N-2r)I_{V_r}.
		\end{align}
		Indeed, evaluating both sides at an \(r\)-set \(S\), the off-diagonal terms cancel, while the coefficient of \(f(S)\) is \((N-r)-r=N-2r\).
		
		Set \(W_\ell^\ell=H_E^\ell\) and \(W_{r+1}^\ell=U_rW_r^\ell\).  A direct induction using \eqref{eqn:difference-identity} gives, for every \(\ell\le r<N-\ell\),
		\begin{align}\label{eqn:const-identity}
			U_r^*U_r
			=
			D_{r+1}U_r
			=
			(r+1-\ell)(N-r-\ell)I
			\quad\text{on }W_r^\ell .
		\end{align}
		Since the scalar in \eqref{eqn:const-identity} is positive, \(U_r\) has full column rank on \(W_r^\ell\).  Hence the successive lifts are injective up to level \(N-\ell\).  In particular, the lift \(h\mapsto h^{(a)}\) is injective on \(H_E^\ell\) whenever \(\ell\le\min(a,N-a)\).
		
		Also, since \(D_\ell=U_{\ell-1}^*\), the injectivity of \(U_{\ell-1}\) for \(\ell\le N/2\) implies that \(D_\ell:V_\ell\to V_{\ell-1}\) is surjective.  Therefore
		\[
		\dim H_E^\ell
		=
		\dim\ker D_\ell
		=
		\binom{N}{\ell}-\binom{N}{\ell-1},
		\qquad 0\le \ell\le N/2,
		\]
		with \(\binom{N}{-1}=0\).  Hence
		\[
		\sum_{\ell=0}^{\min(a,N-a)}
		\dim H_E^\ell
		=
		\binom{N}{\min(a,N-a)}
		=
		\binom{N}{a}.
		\]
		
		It remains to prove orthogonality.  Let \(\ell<\ell'\), \(h\in H_E^\ell\), and \(g\in H_E^{\ell'}\).  Here \(U_{\ell\to a}\) denotes the iterated up map from level \(\ell\) to level \(a\), so \(U_{\ell\to a}h\) is a positive scalar multiple of \(h^{(a)}\).  By adjointness and \eqref{eqn:const-identity},
		\[
		\left\langle U_{\ell\to a}h,\ U_{\ell'\to a}g\right\rangle
		=
		\left\langle
		D_{\ell'+1}\cdots D_a\,U_{\ell\to a}h,\ g
		\right\rangle
		=
		c\left\langle U_{\ell\to\ell'}h,\ g\right\rangle
		\]
		for some scalar \(c\).  Applying adjointness once more,
		\[
		\left\langle U_{\ell\to\ell'}h,\ g\right\rangle
		=
		\left\langle h,\ D_{\ell+1}\cdots D_{\ell'}g\right\rangle
		=0,
		\]
		because \(g\in\ker D_{\ell'}\).  Thus the lifted spaces for different \(\ell\) are orthogonal.
		
		The dimension count above shows that these orthogonal injective images have total dimension \(\binom{N}{a}\), which is the dimension of \(V_a=L^2\!\left(\binom{E}{a}\right)\).  Hence they fill \(V_a\), proving the claimed orthogonal decomposition.
	\end{proof}

	\begin{lemma}[Disjointness averaging]\label{lem:disjoint-average}
		Let \(h\in H_E^\ell\).  If \(|A|=a\), and \(B\), given $A$, is chosen uniformly among the \(b\)-subsets of \(E\setminus A\), then
		\[
		\E[h^{(b)}(B)\mid A]
		=(-1)^\ell\frac{(b)_\ell}{(N-a)_\ell}h^{(a)}(A),
		\]
		with the convention that the right side is zero if the denominator vanishes.  Here \((r)_\ell=r(r-1)\cdots(r-\ell+1)\).
	\end{lemma}
	
	\begin{proof}
		When \((N-a)_\ell>0\), averaging over \(B\subseteq E\setminus A\) gives
		\begin{equation} \label{eq:hbBgivenA}
			\E[h^{(b)}(B)\mid A] = {N-a \choose b}^{-1} \sum_{B \subseteq E \setminus A, \, |B|=b} h^{(b)}(B)
			=\frac{(b)_\ell}{(N-a)_\ell}
			\sum_{L\subseteq E\setminus A,\ |L|=\ell}h(L).
		\end{equation}
		On the set $\{L \subset E: \, |L| = \ell\}$, define the measure $\mu(\{L\}) = h(L)$.
		For $u \in A$, let $C_u = \{L \subset E: \, |L| = \ell, \, u \in L\}$.
		By inclusion-exclusion,
		\[
		\mu\left( \bigcup_{u \in A} C_u \right) = \sum_{j=1}^a (-1)^{j-1} \sum_{J: \, |J| = j} \mu\left( \bigcap_{u \in J} C_u \right).
		\]
		That is,
		\[
		\sum_{L: \, L \cap A \neq \emptyset} h(L) = \sum_{j=1}^a (-1)^{j-1} \sum_{J: \, |J| = j} \sum_{L \supseteq J, \, |L| = \ell} h(L) = \sum_{J \subseteq A, \, |J| \geq 1} (-1)^{|J| - 1} \sum_{L \supseteq J, \, |L| = \ell} h(L).
		\]
		Equivalently,
		\[
		\sum_{L\subseteq E\setminus A,\ |L|=\ell}h(L)
		=
		\sum_{J\subseteq A}(-1)^{|J|}
		\sum_{L\supseteq J,\ |L|=\ell}h(L).
		\]
		All terms with \(|J|<\ell\) vanish by the lower-shadow condition.  Only terms with \(|J|=\ell\) remain, giving
		\[
		\sum_{L\subseteq E\setminus A,\ |L|=\ell}h(L)
		=(-1)^\ell\sum_{J\subseteq A,\ |J|=\ell}h(J)
		=(-1)^\ell h^{(a)}(A).
		\]
		Combining this with \eqref{eq:hbBgivenA} gives the desired relation.
		If \((N-a)_\ell=0\), then no \(b\)-subset of \(E\setminus A\) can carry a nonzero \(\ell\)-lift in the active cases, and both sides are zero by convention.  This proves the claim.
	\end{proof}
	
	\subsubsection{Generated sectors and the coefficient matrices}
	
	Let
	\[
	R=L^2(\mathcal F_1)+L^2(\mathcal F_2)+L^2(\mathcal F_3).
	\]
	If \(f\perp R\), then \(P_if=0\) for each \(i\), so \(S=P_1+P_2+P_3\) vanishes on \(R^\perp\).  It is therefore enough to analyze \(S\) on \(R\).
	
	For a feasible value \(s=|B_i|\), write
	\[
	a_i(s)=s+d_i.
	\]
	Fix harmonic degrees \((\ell,m)\).  Put
	\[
	V_{\ell,m}=H_P^\ell\otimes H_Q^m.
	\]
	Let \(I_{\ell,m}\) be the set of active indices \((i,s)\) such that the \((\ell,m)\)-layer is present in functions of \((A_i,B_i)\), namely
	\[
	\ell\le\min(a_i(s),p-a_i(s)),
	\qquad
	m\le\min(s,q-s).
	\]
	For \((i,s)\in I_{\ell,m}\), \(h\in H_P^\ell\), and \(g\in H_Q^m\), define
	\[
	F_{i,s}^{h,g}(A,B)
	=\1_{\{|B_i|=s\}}h^{(a_i(s))}(A_i)g^{(s)}(B_i).
	\]
	Let \(R_{\ell,m}\) be the span of these functions over all \((i,s)\in I_{\ell,m}\) and all \(h\otimes g\in V_{\ell,m}\).
	
	Now fix \(i\neq j\).  Given \(A_i,B_i\), row \(j\) is obtained by choosing
	\[
	A_j\subseteq P\setminus A_i,
	\qquad
	B_j\subseteq Q\setminus B_i,
	\]
	with sizes \(|A_j|=a_j(t)=t+d_j\) and \(|B_j|=t\), where \(t\) ranges over feasible values.  The probability of a given \(t\), conditioned on \(|B_i|=s\), is
	\[
	q_{i,s}^{j,t}
	=
	\frac{\binom{p-a_i(s)}{a_j(t)}\binom{q-s}{t}}{N_i(s)},
	\]
	where \(N_i(s)\) is the total number of completions after row \(i\) is frozen.  Using Lemma~\ref{lem:disjoint-average} for \(P\) and \(Q\) gives, for \((j,t)\in I_{\ell,m}\),
	\begin{equation}\label{eq:projection-formula}
		P_iF_{j,t}^{h,g}
		=
		\sum_{s:(i,s)\in I_{\ell,m}}
		(-1)^{\ell+m}q_{i,s}^{j,t}
		\frac{(a_j(t))_\ell}{(p-a_i(s))_\ell}
		\frac{(t)_m}{(q-s)_m}
		F_{i,s}^{h,g}.
	\end{equation}
	Also \(P_jF_{j,t}^{h,g}=F_{j,t}^{h,g}\).  Thus, on the generated sector, the operator
	\[
	S=P_1+P_2+P_3
	\]
	is intertwined with
	\[
	M=I+(-1)^{\ell+m}C,
	\]
	where \(C\) acts on the multiplicity indices \((i,s)\in I_{\ell,m}\), has zero diagonal, and has nonnegative off-diagonal entries
	\begin{equation}\label{eq:C-entry}
		C_{i,s;j,t}
		=q_{i,s}^{j,t}
		\frac{(a_j(t))_\ell}{(p-a_i(s))_\ell}
		\frac{(t)_m}{(q-s)_m},
		\qquad i\neq j.
	\end{equation}
	Equivalently, on the full coefficient space \(\bigoplus_{(i,s)\in I_{\ell,m}}V_{\ell,m}\), the operator is \(M\otimes\Id_{V_{\ell,m}}\).
	
	\begin{lemma}[Sector decomposition of \(R\)]\label{lem:sector-decomp}
		The spaces \(R_{\ell,m}\) form an orthogonal direct sum
		\[
		R=\bigoplus_{\ell,m}^{\perp}R_{\ell,m}.
		\]
		Each \(R_{\ell,m}\) is invariant under every \(P_i\), and hence under \(S=P_1+P_2+P_3\).  The space \(R_{0,0}\) is the additive scalar count sector
		\[
		R_{0,0}=\sum_{i=1}^3L^2(|B_i|),
		\]
		viewed as a subspace of the full count-dependent functions.
	\end{lemma}
	
	\begin{proof}
		For each feasible pair \((i,s)\), define
		\[
		W_{i,s}=\{\1_{\{|B_i|=s\}}\varphi(A_i,B_i)\}\subseteq L^2(\mathcal F_i),
		\]
		where \(\varphi\) ranges over functions on pairs \((A_i,B_i)\) with \(|B_i|=s\) and \(|A_i|=a_i(s)\).  Then
		\[
		L^2(\mathcal F_i)=\bigoplus_s W_{i,s},
		\qquad
		R=\sum_{i,s}W_{i,s}.
		\]
		For fixed \((i,s)\), Lemma~\ref{lem:johnson-decomp} applied to \(P\) and \(Q\) gives
		\[
		W_{i,s}=\bigoplus_{\ell,m}^{\perp}W_{i,s}^{\ell,m},
		\]
		where \(W_{i,s}^{\ell,m}\) is the span of the functions \(F_{i,s}^{h,g}\).  Hence \(R=\sum_{\ell,m}R_{\ell,m}\).
		
		It remains to prove orthogonality across different degree pairs.  Take generators
		\[
		F_{i,s}^{h,g}\in W_{i,s}^{\ell,m},
		\qquad
		F_{j,t}^{h',g'}\in W_{j,t}^{\ell',m'},
		\]
		with \((\ell,m)\neq(\ell',m')\).  If \(i=j\), then either \(s\neq t\), in which case the supports are disjoint, or \(s=t\), in which case orthogonality follows from the tensor product of the Johnson decompositions on \(P\) and \(Q\).
		
		Suppose now that \(i\neq j\).  Since \(F_{i,s}^{h,g}\) is \(\mathcal F_i\)-measurable and \(P_i\) is the orthogonal projection onto \(L^2(\mathcal F_i)\),
		\[
		\ip{F_{i,s}^{h,g}}{F_{j,t}^{h',g'}}
		=
		\ip{F_{i,s}^{h,g}}{P_iF_{j,t}^{h',g'}}.
		\]
		Formula~\eqref{eq:projection-formula}, applied to the degree pair \((\ell',m')\), shows that \(P_iF_{j,t}^{h',g'}\) is a linear combination of functions \(F_{i,s'}^{h',g'}\) with the same degree pair \((\ell',m')\).  Only the term \(s'=s\) can have nonzero inner product with \(F_{i,s}^{h,g}\), and that inner product is zero by fixed-row Johnson orthogonality because \((\ell,m)\neq(\ell',m')\).  Hence \(R_{\ell,m}\perp R_{\ell',m'}\).
		
		The same conditional expectation formula proves invariance.  If \(F_{j,t}^{h,g}\in R_{\ell,m}\), then \(P_jF_{j,t}^{h,g}=F_{j,t}^{h,g}\), while for \(i\neq j\), \(P_iF_{j,t}^{h,g}\) is a linear combination of generators of \(R_{\ell,m}\).  Thus \(P_iR_{\ell,m}\subseteq R_{\ell,m}\) for every \(i\).
		
		Finally, when \((\ell,m)=(0,0)\), the lifted functions are constants on fixed count blocks, so the resulting space is exactly \(\sum_i L^2(|B_i|)\).  This is contained in the full space of functions of \((|B_1|,|B_2|,|B_3|)\).
	\end{proof}
	
	\begin{lemma}[Row-substochastic bound]\label{lem:row-substochastic}
		For every \((\ell,m)\neq(0,0)\), every row sum of \(C\) is at most \(1\).  Consequently \(\rho(C)\le1\).
	\end{lemma}
	
	\begin{proof}
		Fix an output index \((i,s)\in I_{\ell,m}\), and let \(j,k\) be the two rows different from \(i\).  Conditional on row \(i\), the remaining counts satisfy
		\[
		a_j+a_k=p-a_i(s),
		\qquad
		b_j+b_k=q-s.
		\]
		The row sum of \(C\) is at most the same sum extended over all feasible counts of the two remaining rows, namely
		\[
		\E\left[
		\frac{(a_j)_\ell(b_j)_m+(a_k)_\ell(b_k)_m}
		{(a_j+a_k)_\ell(b_j+b_k)_m}
		\ \middle|\ |B_i|=s
		\right],
		\]
		with the convention that the ratio is zero when the denominator is zero.  For active \((i,s)\), this convention is harmless.
		
		For nonnegative integers \(x,y,u,v\), and \((\ell,m)\neq(0,0)\),
		\[
		(x)_\ell(u)_m+(y)_\ell(v)_m
		\le
		(x+y)_\ell(u+v)_m.
		\]
		The right side counts ordered choices of \(\ell\) distinct plus labels and \(m\) distinct minus labels from the union of two parts of sizes \((x,u)\) and \((y,v)\).  The two terms on the left count only the choices lying entirely in the first part or entirely in the second part.  Since at least one of \(\ell,m\) is positive, these two classes are disjoint subfamilies of the choices counted by the right side.  Therefore every row sum is at most \(1\), and the standard matrix norm bound gives \(\rho(C)\le1\).
	\end{proof}
	
	\begin{lemma}[Quotient/intertwining lemma]\label{lem:quotient}
		Let \(U,W\) be finite-dimensional real inner-product spaces.  Let \(T:U\to W\) be a linear map, let \(A:W\to W\) be a self-adjoint operator preserving \(\im T\), and let \(M=I+\epsilon C\) be an operator on \(U\), with \(\epsilon\in\{1,-1\}\), such that
		\[
		AT=TM.
		\]
		If \(\rho(C)\le1\), then every eigenvalue of \(A\) on \(\im T\) is at most \(2\).
	\end{lemma}
	
	\begin{proof}
		The relation \(AT=TM\) implies that \(\ker T\) is \(M\)-invariant.  Therefore \(M\) induces an operator \(\overline M\) on the quotient \(U/\ker T\), and \(T\) identifies this quotient with \(\im T\).  Under this identification, \(A|_{\im T}\) is exactly \(\overline M\).
		
		Choose a basis of \(U\) adapted to \(\ker T\).  The matrix of \(M\) is block upper triangular, with one diagonal block giving the action on \(\ker T\) and the other diagonal block giving \(\overline M\).  Hence every eigenvalue of \(\overline M\) is an eigenvalue of \(M\).  Since \(M=I+\epsilon C\) and \(\rho(C)\le1\), every eigenvalue of \(M\) has the form \(1+
		\epsilon\mu\) with \(|\mu|\le1\).  Because \(A\) is self-adjoint and \(\im T\) is invariant, the restriction \(A|_{\im T}\) is self-adjoint; its eigenvalues are real.  Thus, for such an eigenvalue \(\lambda\), \(\lambda=1+\epsilon\mu\) with \(|\mu|\le1\) and \(\lambda\in\mathbb R\), which implies \(\lambda\le2\).
	\end{proof}
	
	\begin{proposition}[Non-scalar harmonic sectors]\label{prop:nonscalar}
		On every harmonic sector \(R_{\ell,m}\) with \((\ell,m)\neq(0,0)\), all eigenvalues of
		\[
		K_3=\frac13(P_1+P_2+P_3)
		\]
		are at most \(2/3\).
	\end{proposition}
	
	\begin{proof}
		Let
		\[
		U_{\ell,m}=\bigoplus_{(i,s)\in I_{\ell,m}}V_{\ell,m},
		\]
		and define \(T_{\ell,m}:U_{\ell,m}\to L^2\) by
		\[
		T_{\ell,m}(e_{i,s}\otimes(h\otimes g))=F_{i,s}^{h,g}.
		\]
		By definition, \(\im T_{\ell,m}=R_{\ell,m}\).  Lemma~\ref{lem:sector-decomp} gives that \(R_{\ell,m}\) is invariant under \(S=P_1+P_2+P_3\), and \(S|_{R_{\ell,m}}\) is self-adjoint.  Formula~\eqref{eq:projection-formula} gives the exact intertwining relation
		\[
		ST_{\ell,m}
		=
		T_{\ell,m}\left((I+(-1)^{\ell+m}C)\otimes\Id_{V_{\ell,m}}\right).
		\]
		Lemma~\ref{lem:row-substochastic} gives \(\rho(C)\le1\), hence the same spectral-radius bound for \(C\otimes\Id_{V_{\ell,m}}\).  Applying Lemma~\ref{lem:quotient} with \(A=S|_{R_{\ell,m}}\) shows that every eigenvalue of \(S\) on \(R_{\ell,m}\) is at most \(2\).  Dividing by \(3\) proves the claim for \(K_3=S/3\).
	\end{proof}
	
	\subsection{Proof of the three-row inequality}
	
	\begin{proof}[Proof of Theorem~\ref{thm:three-row}]
		The operator \(S=P_1+P_2+P_3\) is self-adjoint.  Since
		\[
		R=L^2(\mathcal F_1)+L^2(\mathcal F_2)+L^2(\mathcal F_3),
		\]
		we have \(P_if=0\) for all \(i\) whenever \(f\perp R\).  Thus \(K_3=S/3\) vanishes on \(R^\perp\), and all nonzero eigenvalues occur inside \(R\).
		
		By Lemma~\ref{lem:sector-decomp},
		\[
		R=\bigoplus_{\ell,m}^{\perp}R_{\ell,m},
		\]
		and each sector is invariant under \(K_3\).  On the additive scalar count sector \(R_{0,0}\), Lemmas~\ref{lem:bessel} and~\ref{lem:count-id} give nonconstant eigenvalues at most \(2/3\), because \(R_{0,0}\) is an invariant subspace of the full count-dependent chain.  On every non-scalar sector, Proposition~\ref{prop:nonscalar} gives the same upper bound.  The only eigenvalue larger than \(2/3\) is the constant eigenvalue \(1\).  Hence \(\operatorname{gap}(K_3)\ge1/3\).
		
		Finally,
		\[
		\ip{f}{K_3f}
		=\frac13\sum_{i=1}^3\ip{f}{P_if}
		=\frac13\sum_{i=1}^3\norm{P_if}_2^2.
		\]
		For mean-zero \(f\), \(\norm{P_if}_2^2=\Var(\E[f\mid\mathcal F_i])\).  The eigenvalue bound \(K_3\le(2/3)I\) on the mean-zero space is therefore equivalent to
		\[
		\sum_{i=1}^3\Var(\E[f\mid\mathcal F_i])\le2\Var(f).
		\]
	\end{proof}
	
\section{Proofs of the Main Theorem and  Corollary \ref{cor:MV-fixed-margin}}
	
	\begin{proof}[Proof of Theorem~\ref{thm:main}]
		First assume \(m\ge3\). Fix any triple of rows and condition on all rows outside this triple. The remaining configurations form a feasible \(3\times n\) fixed-margin state space. Applying Theorem~\ref{thm:three-row} with \(m=3\), the corresponding row-pair heat-bath chain has spectral gap at least \(1/3\). Proposition~\ref{prop:connectivity} implies that the two-row heat-bath chain is irreducible, because every single swap is contained in the support of the corresponding two-row heat-bath update.  Lemma~\ref{lem:three-row-reduction}, applied with \(M=m\), gives
		\[
		\gamma_H\ge \binom{m}{2}^{-1}.
		\]
		Then Lemma~\ref{lem:two-column} gives
		\[
		\gamma_{\rm sw}
		\ge \binom{n}{2}^{-1}\gamma_H
		\ge \binom{m}{2}^{-1}\binom{n}{2}^{-1}.
		\]
		
		For \(m=2\), the row-pair heat-bath update is a one-step perfect sampler and therefore has spectral gap \(1\), so Lemma~\ref{lem:two-column} yields
		\[
		\gamma_{\rm sw}
		\ge \binom n2^{-1}
		=
		\binom22^{-1}\binom n2^{-1}.
		\]
		
		The mixing-time estimate follows from the standard reversible chain bound (Theorem 12.4 of \cite{levin2017markov})
		\[
		t_{\rm mix}(\varepsilon)
		\le \gamma_{\rm sw}^{-1}
		\left(\log\frac1{\pi_{\min}}+\log\varepsilon^{-1}\right),
		\]
		with \(\pi_{\min}=1/|\Omega(r,c)|\), and from \(|\Omega(r,c)|\le2^{mn}\).
	\end{proof}
\begin{proof}[Proof of Corollary~\ref{cor:MV-fixed-margin}]
	Let
	\[
	\Omega=\Omega(r,c)
	\]
	and let \(G_{\rm sw}\) be the simple graph on \(\Omega\) in which two matrices
	are adjacent if and only if they differ by one \(2\times2\) swap.  If
	\(|\Omega|\le1\), then the expansion statement is vacuous.  If \(m<2\) or
	\(n<2\), then a feasible matrix is uniquely determined by its margins, so the
	statement is again vacuous.  Hence assume \(m,n\ge2\) and \(|\Omega|\ge2\).
	
	Put
	\[
	M=\binom m2\binom n2.
	\]
	Let \(P_{\rm sw}^{\rm nl}\) be the corresponding non-lazy swap chain: at each
	step it chooses one row pair and one column pair uniformly, performs the swap
	if the chosen \(2\times2\) submatrix is switchable, and otherwise stays put.
	Then the lazy swap chain satisfies
	\[
	P_{\rm sw}=\frac12 I+\frac12 P_{\rm sw}^{\rm nl}.
	\]
	Therefore, for every function \(f\),
	\[
	\mathcal E_{\rm sw}^{\rm nl}(f,f)
	=
	2\mathcal E_{\rm sw}(f,f).
	\]
	By Theorem~\ref{thm:main},
	\[
	\mathcal E_{\rm sw}(f,f)
	\ge
	\frac1M \Var_\pi(f),
	\]
	where \(\pi\) is the uniform measure on \(\Omega\).  Hence the non-lazy chain
	satisfies
	\[
	\mathcal E_{\rm sw}^{\rm nl}(f,f)
	\ge
	\frac2M \Var_\pi(f).
	\]
	
	Now take \(f=\1_U\), where \(U\subseteq\Omega\).  Since
	\(P_{\rm sw}^{\rm nl}(X,Y)=1/M\) whenever \(X\) and \(Y\) are joined by a swap
	edge, we have
	\[
	\mathcal E_{\rm sw}^{\rm nl}(\1_U,\1_U)
	=
	\frac{|\partial_{G_{\rm sw}}U|}{M|\Omega|}.
	\]
	Also
	\[
	\Var_\pi(\1_U)
	=
	\frac{|U|}{|\Omega|}\left(1-\frac{|U|}{|\Omega|}\right).
	\]
	Thus
	\[
	\frac{|\partial_{G_{\rm sw}}U|}{M|\Omega|}
	\ge
	\frac2M
	\frac{|U|}{|\Omega|}\left(1-\frac{|U|}{|\Omega|}\right).
	\]
	Equivalently,
	\[
	|\partial_{G_{\rm sw}}U|
	\ge
	2|U|\left(1-\frac{|U|}{|\Omega|}\right).
	\]
	If \(0<|U|\le |\Omega|/2\), this gives
	\[
	|\partial_{G_{\rm sw}}U|\ge |U|.
	\]
	Therefore the unweighted swap graph has edge expansion at least one.
	
	It remains to relate the swap graph to the graph of the polytope \(P_{r,c}\).
	First, every point of \(\Omega\) is a vertex of \(P_{r,c}\).  Indeed, if
	\(X\in\Omega\) were written as a convex combination of points of \(\Omega\),
	then each coordinate equal to \(1\) in \(X\) would have to be \(1\) in every
	point appearing with positive weight, and each coordinate equal to \(0\) in
	\(X\) would have to be \(0\) in every point appearing with positive weight.
	Thus every point appearing in the combination is equal to \(X\).
	
	Next, suppose \(X,Y\in\Omega\) differ by one \(2\times2\) swap.  After
	renaming rows and columns, assume that the swap uses rows \(a,b\) and columns
	\(u,v\), with
	\[
	X_{au}=X_{bv}=1,\qquad X_{av}=X_{bu}=0,
	\]
	and with \(Y\) obtained from \(X\) by switching these four entries.  Let
	\[
	D^+=\{(i,j):X_{ij}=Y_{ij}=1\},
	\qquad
	D^-=\{(i,j):X_{ij}=Y_{ij}=0\}.
	\]
	Consider the linear functional
	\[
	L(Z)=\sum_{(i,j)\in D^+} Z_{ij}
	-\sum_{(i,j)\in D^-} Z_{ij}.
	\]
	For every \(Z\in\Omega\), we have \(L(Z)\le |D^+|\).  Equality holds only if
	\(Z\) agrees with \(X\) and \(Y\) outside the four swapped entries.  Under this
	agreement, the prescribed row and column sums force the remaining \(2\times2\)
	block on rows \(a,b\) and columns \(u,v\) to have both row sums and both column
	sums equal to one.  The only two binary \(2\times2\) matrices with these row
	and column sums are
	\[
	\begin{pmatrix}1&0\\0&1\end{pmatrix}
	\quad\text{and}\quad
	\begin{pmatrix}0&1\\1&0\end{pmatrix}.
	\]
	Thus the face of \(P_{r,c}\) exposed by \(L\) has exactly the two vertices
	\(X\) and \(Y\).  Hence \([X,Y]\) is an edge of \(P_{r,c}\).
	
	Therefore the graph of \(P_{r,c}\) contains \(G_{\rm sw}\) as a spanning
	subgraph.  Adding edges cannot decrease edge expansion, so the graph of
	\(P_{r,c}\) also has edge expansion at least one.  This proves the
	Mihail--Vazirani conjecture for fixed-margin \(0/1\)-matrix polytopes.
\end{proof}

\section*{Acknowledgement} The authors thank Persi Diaconis, Péter L. Erdős, Catherine Greenhill, Mark Sellke, and Prasad Tetali for helpful comments.
	\bibliographystyle{chicago}
	\bibliography{ref}

@article{neal2024pattern,
	title={Pattern detection in bipartite networks: a review of terminology, applications, and methods},
	author={Neal, Zachary P and Cadieux, Annabell and Garlaschelli, Diego and Gotelli, Nicholas J and Saracco, Fabio and Squartini, Tiziano and Shutters, Shade T and Ulrich, Werner and Wang, Guanyang and Strona, Giovanni},
	journal={PLOS Complex Systems},
	volume={1},
	number={2},
	pages={e0000010},
	year={2024},
	publisher={Public Library of Science San Francisco, CA USA}
}

@book{rasch1993probabilistic,
	title={Probabilistic models for some intelligence and attainment tests.},
	author={Rasch, Georg},
	year={1993},
	publisher={ERIC}
}

@article{besag1989generalized,
	title={Generalized monte carlo significance tests},
	author={Besag, Julian and Clifford, Peter},
	journal={Biometrika},
	volume={76},
	number={4},
	pages={633--642},
	year={1989},
	publisher={Oxford University Press}
}

@article{ryser1957combinatorial,
	title={Combinatorial properties of matrices of zeros and ones},
	author={Ryser, Herbert J},
	journal={Canadian Journal of Mathematics},
	volume={9},
	pages={371--377},
	year={1957},
	publisher={Cambridge University Press}
}

@inproceedings{kannan1997simple,
	title={Simple Markov-chain algorithms for generating bipartite graphs and tournaments},
	author={Kannan, Ravi and Tetali, Prasad and Vempala, Santosh},
	booktitle={Proceedings of the eighth annual ACM-SIAM symposium on Discrete algorithms},
	pages={193--200},
	year={1997}
}

@incollection{gale1957theorem,
	title={A theorem on flows in networks},
	author={Gale, David},
	booktitle={Classic Papers in Combinatorics},
	pages={259--268},
	year={1957},
	publisher={Springer}
}

@article{erdHos2022mixing,
	title={The mixing time of switch Markov chains: a unified approach},
	author={Erd{\H{o}}s, P{\'e}ter L and , Catherine and Mezei, Tam{\'a}s R{\'o}bert and Mikl{\'o}s, Istv{\'a}n and Solt{\'e}sz, D{\'a}niel and Soukup, Lajos},
	journal={European Journal of Combinatorics},
	volume={99},
	pages={103421},
	year={2022},
	publisher={Elsevier}
}

@inproceedings{amanatidis2019rapid,
	title={Rapid mixing of the switch Markov chain for strongly stable degree sequences and 2-class joint degree matrices},
	author={Amanatidis, Georgios and Kleer, Pieter},
	booktitle={Proceedings of the Thirtieth Annual ACM-SIAM Symposium on Discrete Algorithms},
	pages={966--985},
	year={2019},
	organization={SIAM}
}

@article{gao2021mixing,
	title={Mixing time of the switch Markov chain and stable degree sequences},
	author={Gao, Pu and Greenhill, Catherine},
	journal={Discrete Applied Mathematics},
	volume={291},
	pages={143--162},
	year={2021},
	publisher={Elsevier}
}

@article{jerrum1990fast,
	title={Fast uniform generation of regular graphs},
	author={Jerrum, Mark and Sinclair, Alistair},
	journal={Theoretical Computer Science},
	volume={73},
	number={1},
	pages={91--100},
	year={1990},
	publisher={Elsevier}
}

@article{strona2014fast,
	title={A fast and unbiased procedure to randomize ecological binary matrices with fixed row and column totals},
	author={Strona, Giovanni and Nappo, Domenico and Boccacci, Francesco and Fattorini, Simone and San-Miguel-Ayanz, Jesus},
	journal={Nature Communications},
	volume={5},
	number={1},
	pages={4114},
	year={2014},
	publisher={Nature Publishing Group UK London}
}

@article{caputo2008spectral,
	title={On the spectral gap of the Kac walk and other binary collision processes},
	author={Caputo, Pietro},
	journal={Alea},
	volume={4},
	pages={205--222},
	year={2008}
}

@inproceedings{carstens2018speeding,
	title={Speeding up switch Markov chains for sampling bipartite graphs with given degree sequence},
	author={Carstens, Corrie Jacobien and Kleer, Pieter},
	booktitle={Approximation, Randomization, and Combinatorial Optimization. Algorithms and Techniques (APPROX/RANDOM 2018)},
	pages={36--1},
	year={2018},
	organization={Schloss Dagstuhl--Leibniz-Zentrum f{\"u}r Informatik}
}

@article{greenhill2018switch,
	title={The switch Markov chain for sampling irregular graphs and digraphs},
	author={Greenhill, Catherine and Sfragara, Matteo},
	journal={Theoretical Computer Science},
	volume={719},
	pages={1--20},
	year={2018},
	publisher={Elsevier}
}

@inproceedings{bubley1997path,
	title={Path coupling: A technique for proving rapid mixing in Markov chains},
	author={Bubley, Russ and Dyer, Martin},
	booktitle={Proceedings 38th Annual Symposium on Foundations of Computer Science},
	pages={223--231},
	year={1997},
	organization={IEEE}
}

@book{levin2017markov,
	author    = {Levin, David A. and Peres, Yuval},
	title     = {Markov Chains and Mixing Times},
	edition   = {Second},
	publisher = {American Mathematical Society},
	address   = {Providence, RI},
	year      = {2017},
	pages     = {447},
	doi       = {10.1090/mbk/107},
	isbn      = {978-1-4704-2962-1}
}

@inproceedings{mihail1992expansion,
	title={On the expansion of combinatorial polytopes},
	author={Mihail, Milena},
	booktitle={International Symposium on Mathematical Foundations of Computer Science},
	pages={37--49},
	year={1992},
	organization={Springer}
}

@article{filmus2016orthogonal,
	title={An Orthogonal Basis for Functions over a Slice of the Boolean Hypercube},
	author       = {Yuval Filmus},
	journal      = {The Electronic Journal of Combinatorics},
	volume       = {23},
	number       = {1},
	pages        = {1},
	year         = {2016},
	url          = {https://doi.org/10.37236/4567},
	doi          = {10.37236/4567},
	bibsource    = {dblp computer science bibliography, https://dblp.org}
}

@book{sagan2001symmetric,
	title={The symmetric group: representations, combinatorial algorithms, and symmetric functions},
	author={Sagan, Bruce},
	volume={203},
	year={2001},
	publisher={Springer Science \& Business Media}
}

@article{miklos2013towards,
	title={Towards Random Uniform Sampling of Bipartite Graphs with given Degree Sequence},
	volume={20}, journal={The Electronic Journal of Combinatorics}, author={Miklós, István and Erdős, Péter L and Soukup, Lajos}, year={2013}, pages={P16} }

@article{erdHos2018new,
	title={New classes of degree sequences with fast mixing swap Markov chain sampling},
	author={Erd{\H{o}}s, P{\'e}ter L and Mikl{\'o}s, Istv{\'a}n and Toroczkai, Zolt{\'a}n},
	journal={Combinatorics, Probability and Computing},
	volume={27},
	number={2},
	pages={186--207},
	year={2018},
	publisher={Cambridge University Press}
}

@article{cooper2007sampling,
	title={Sampling regular graphs and a peer-to-peer network},
	author={Cooper, Colin and Dyer, Martin and Greenhill, Catherine},
	journal={Combinatorics, Probability and Computing},
	volume={16},
	number={4},
	pages={557--593},
	year={2007},
	publisher={Cambridge University Press}
}

@article{greenhill2011polynomial,
	title={A Polynomial Bound on the Mixing Time of a Markov Chain for Sampling Regular Directed Graphs},
	author={Greenhill, Catherine},
	volume={18}, DOI={10.37236/721}, number={1}, journal={The Electronic Journal of Combinatorics}, year={2011}, pages={P234} 
}

@inproceedings{greenhill2014switch,
	title={The switch Markov chain for sampling irregular graphs},
	author={Greenhill, Catherine},
	booktitle={Proceedings of the twenty-sixth annual acm-siam symposium on discrete algorithms},
	pages={1564--1572},
	year={2014},
	organization={SIAM}
}

@article{qin2024spectral,
	title={Spectral gap bounds for reversible hybrid {G}ibbs chains},
	author={Qin, Qian and Ju, Nianqiao and Wang, Guanyang},
	year={2025},
	volume={53},
	pages={1613--1638},
	journal={Annals of Statistics},
	publisher={Institute of Mathematical Statistics}
}

@incollection{kaibel2004expansion,
	title={On the expansion of graphs of 0/1-polytopes},
	author={Kaibel, Volker},
	booktitle={The Sharpest Cut: The Impact of Manfred Padberg and His Work},
	pages={199--216},
	year={2004},
	publisher={SIAM}
}

@article{anari2021spectral,
	title={Spectral independence in high-dimensional expanders and applications to the hardcore model},
	author={Anari, Nima and Liu, Kuikui and Gharan, Shayan Oveis},
	journal={SIAM Journal on Computing},
	volume={53},
	number={6},
	pages={FOCS20--1},
	year={2021},
	publisher={SIAM}
}

@inproceedings{blanca2022mixing,
	title={On mixing of markov chains: Coupling, spectral independence, and entropy factorization},
	author={Blanca, Antonio and Caputo, Pietro and Chen, Zongchen and Parisi, Daniel and {\v{S}}tefankovi{\v{c}}, Daniel and Vigoda, Eric},
	booktitle={Proceedings of the 2022 Annual ACM-SIAM Symposium on Discrete Algorithms (SODA)},
	pages={3670--3692},
	year={2022},
	organization={SIAM}
}

@article{feng2022rapid,
	title={Rapid mixing from spectral independence beyond the Boolean domain},
	author={Feng, Weiming and Guo, Heng and Yin, Yitong and Zhang, Chihao},
	journal={ACM Transactions on Algorithms (TALG)},
	volume={18},
	number={3},
	pages={1--32},
	year={2022},
	publisher={ACM New York, NY}
}

@book{liu2023spectral,
	title={Spectral Independence a New Tool to Analyze Markov Chains},
	author={Liu, Kuikui},
	year={2023},
	publisher={University of Washington}
}

@article{erdHos2018efficiently,
	title={Efficiently sampling the realizations of bounded, irregular degree sequences of bipartite and directed graphs},
	author={Erd{\H{o}}s, P{\'e}ter L and Mezei, Tam{\'a}s R{\'o}bert and Mikl{\'o}s, Istv{\'a}n and Solt{\'e}sz, D{\'a}niel},
	journal={Plos one},
	volume={13},
	number={8},
	pages={e0201995},
	year={2018},
	publisher={Public Library of Science San Francisco, CA USA}
}

@article{tikhomirov2024regularized,
	title={Regularized modified log-{S}obolev inequalities and comparison of {M}arkov chains},
	author={Tikhomirov, Konstantin and Youssef, Pierre},
	journal={The Annals of Probability},
	volume={52},
	number={4},
	pages={1201--1224},
	year={2024},
	publisher={Institute of Mathematical Statistics}
}

@article{tikhomirov2023sharp,
	title={Sharp {P}oincar{\'e} and log-{S}obolev inequalities for the switch chain on regular bipartite graphs},
	author={Tikhomirov, Konstantin and Youssef, Pierre},
	journal={Probability Theory and Related Fields},
	volume={185},
	number={1-2},
	pages={89--184},
	year={2023},
	publisher={Springer Nature BV}
}

@article{dyer2021sampling,
	title={Sampling hypergraphs with given degrees},
	author={Dyer, Martin and Greenhill, Catherine and Kleer, Pieter and Ross, James and Stougie, Leen},
	journal={Discrete Mathematics},
	volume={344},
	number={11},
	pages={112566},
	year={2021},
	publisher={Elsevier}
}

@article{rao1996markov,
	title={A Markov chain Monte Carlo method for generating random (0, 1)-matrices with given marginals},
	author={Rao, A Ramachandra and Jana, Rabindranath and Bandyopadhyay, Suraj},
	journal={Sankhy{\=a}: The Indian Journal of Statistics, Series A},
	pages={225--242},
	year={1996},
	publisher={JSTOR}
}

@article{diaconis1998algebraic,
	title={Algebraic algorithms for sampling from conditional distributions},
	author={Diaconis, Persi and Sturmfels, Bernd},
	journal={The Annals of statistics},
	volume={26},
	number={1},
	pages={363--397},
	year={1998},
	publisher={Institute of Mathematical Statistics}
}

@article{barvinok2010number,
	title={On the number of matrices and a random matrix with prescribed row and column sums and 0--1 entries},
	author={Barvinok, Alexander},
	journal={Advances in Mathematics},
	volume={224},
	number={1},
	pages={316--339},
	year={2010},
	publisher={Elsevier}
}

@article{canfield2008asymptotic,
	title={Asymptotic enumeration of dense 0--1 matrices with specified line sums},
	author={Canfield, E Rodney and Greenhill, Catherine and McKay, Brendan D},
	journal={Journal of Combinatorial Theory, Series A},
	volume={115},
	number={1},
	pages={32--66},
	year={2008},
	publisher={Elsevier}
}

@article{greenhill2006asymptotic,
	title={Asymptotic enumeration of sparse 0--1 matrices with irregular row and column sums},
	author={Greenhill, Catherine and McKay, Brendan D and Wang, Xiaoji},
	journal={Journal of Combinatorial Theory, Series A},
	volume={113},
	number={2},
	pages={291--324},
	year={2006},
	publisher={Elsevier}
}

@article{jerrum2004polynomial,
	title={A polynomial-time approximation algorithm for the permanent of a matrix with nonnegative entries},
	author={Jerrum, Mark and Sinclair, Alistair and Vigoda, Eric},
	journal={Journal of the ACM (JACM)},
	volume={51},
	number={4},
	pages={671--697},
	year={2004},
	publisher={ACM New York, NY, USA}
}

@article{bezakova2007sampling,
	title={Sampling binary contingency tables with a greedy start},
	author={Bez{\'a}kov{\'a}, Ivona and Bhatnagar, Nayantara and Vigoda, Eric},
	journal={Random Structures \& Algorithms},
	volume={30},
	number={1-2},
	pages={168--205},
	year={2007},
	publisher={Wiley Online Library}
}

@article{delsarte1978hahn,
	title={Hahn polynomials, discrete harmonics, and {$t$}-designs},
	author={Delsarte, Philippe},
	journal={SIAM Journal on Applied Mathematics},
	volume={34},
	number={1},
	pages={157--166},
	year={1978},
	doi={10.1137/0134012}
}

@misc{caracciolo1992two,
	title={Two remarks on simulated tempering},
	author={Caracciolo, Sergio and Pelissetto, Andrea and Sokal, Alan D},
	howpublished={Unpublished manuscript},
	year={1992}
}

@article{madras2002markov,
	title={Markov chain decomposition for convergence rate analysis},
	author={Madras, Neal and Randall, Dana},
	journal={Annals of Applied Probability},
	pages={581--606},
	year={2002},
	publisher={JSTOR}
}

@misc{ge2018simulated,
	title={{Simulated tempering Langevin Monte Carlo II: An improved proof using soft Markov chain decomposition}},
	author={Ge, Rong and Lee, Holden and Risteski, Andrej},
	howpublished={arXiv preprint},
	year={2018}
}

@article{andrieu2018uniform,
	title={{Uniform ergodicity of the iterated conditional SMC and geometric ergodicity of particle Gibbs samplers}},
	author={Andrieu, Christophe and Lee, Anthony and Vihola, Matti},
	year={2018},
	journal={Bernoulli},
	volume={24},
	pages={842--872}
}

@article{anari2024log,
  title={Log-concave polynomials II: High-dimensional walks and an FPRAS for counting bases of a matroid},
  author={Anari, Nima and Liu, Kuikui and Gharan, Shayan Oveis and Vinzant, Cynthia},
  journal={Annals of Mathematics},
  volume={199},
  number={1},
  pages={259--299},
  year={2024},
  publisher={Department of Mathematics, Princeton University Princeton, New Jersey, USA}
}
\end{document}